\documentclass[leqno]{amsart}
\usepackage{amsmath}
\usepackage{amssymb}
\usepackage{amsthm}
\usepackage{enumerate}
\usepackage[mathscr]{eucal}
\theoremstyle{plain}
\usepackage{tikz}
\newtheorem{theorem}{Theorem}[section]
\newtheorem{lemma}[theorem]{Lemma}

\theoremstyle{definition}
\newtheorem{definition}[theorem]{Definition}
\newtheorem{remark}[theorem]{Remark}

\newtheorem{cor}[theorem]{Corollary}
\theoremstyle{remark}

\begin{document}

\title [Birkhoff-James orthogonality to a subspace]{Birkhoff-James orthogonality to a subspace of operators defined between Banach spaces}

\author[Arpita Mal and Kallol Paul ]{Arpita Mal and Kallol Paul}

\address[Mal]{Department of Mathematics\\ Jadavpur University\\ Kolkata 700032\\ India}
\email{arpitamalju@gmail.com}

\address{(Paul)Department of Mathematics, Jadavpur University, Kolkata 700032, India}
\email{kalloldada@gmail.com}

\thanks{Miss Arpita Mal would like to thank UGC, Govt. of India for the financial support. Prof. Kallol Paul would like to thank RUSA 2.0, Jadavpur University for the financial support.}




\date{}
\maketitle
\begin{abstract}
This paper deals with study of Birkhoff-James orthogonality of a linear operator to a subspace of operators defined between arbitrary Banach spaces. In case the domain space is reflexive and the subspace is finite dimensional we obtain a complete characterization. For arbitrary Banach spaces, we obtain the same under some additional conditions. For arbitrary Hilbert space $ \mathbb{H},$ we also study orthogonality to subspace of the space of linear operators $L(\mathbb{H}), $ both with respect to operator norm as well as numerical radius norm.
\end{abstract}

\section{Introduction}
Geometry of Banach space is an enriched area of research. Nowadays Birkhoff-James orthogonality is drawing attention to the mathematicians working in this area because of its importance in the study of geometry of Banach space. Recently in \cite{BS1,P,PHD,SPM}, Birkhoff-James orthogonality has been studied in operator spaces. The purpose of this paper is to obtain characterization of Birkhoff-James orthogonality of a bounded linear operator defined between arbitrary Banach spaces to an arbitrary subspace of operator space.

 Suppose $\mathbb{X},\mathbb{Y}$ denote Banach spaces and $\mathbb{H}$ denotes Hilbert space over $K$ $ (\mathbb{R}~\text{or } \mathbb{C}).$ $B_{\mathbb{X}}$ and $S_{\mathbb{X}}$ denote the unit ball and the unit sphere of $\mathbb{X}$ respectively, i.e., $B_{\mathbb{X}}=\{x\in\mathbb{X}:\|x\|\leq 1\}$ and $S_{\mathbb{X}}=\{x\in\mathbb{X}:\|x\|= 1\}.$ The space of all bounded (compact) linear operators between $\mathbb{X}$ and $\mathbb{Y}$ is denoted by $\mathbb{L}(\mathbb{X},\mathbb{Y}) ~ (\mathbb{K}(\mathbb{X},\mathbb{Y})).$ $\mathbb{L}(\mathbb{H})~(\mathbb{K}(\mathbb{H}))$ denotes the space of all bounded (compact) linear operators from $\mathbb{H}$ to $\mathbb{H}.$ In particular, if $\mathbb{H}=K^n$ with usual inner product on $K^n,$ then we denote $\mathbb{L}(\mathbb{H})$ by $M_n(K).$ $\mathbb{X}^*$ denotes the space of all bounded linear functionals on $\mathbb{X}.$ For $x(\neq 0)\in \mathbb{X},$ let $J(x)=\{f\in S_{\mathbb{X}^*}:f(x)=\|x\|\}$ denote the set of all supporting linear functionals of $x.$ Observe that $J(x)$ is a non-empty, weak* compact, convex subset of $S_{\mathbb{X}^*}.$ If for each $x\in S_{\mathbb{X}},~J(x)$ is singleton, then $\mathbb{X}$ is said to be a smooth space. For a convex set $A,$ the set of all extreme points of $A$ is denoted by $Ext(A).$  If $x,y\in \mathbb{X},$ then  $x$ is said to be  Birkhoff-James  orthogonal  \cite{B,J}  to $y,$ written as $x\perp_B y,$ if $\|x+\lambda y\|\geq \|x\|$ for all $\lambda \in K.$ For $T,A\in \mathbb{L}(\mathbb{X},\mathbb{Y}),~T\perp_B A$ if  $\|T+\lambda A\|\geq \|T\|$ for all $\lambda \in K.$  If $T\in \mathbb{L}(\mathbb{X},\mathbb{Y})$ and $\mathcal{W}$ is a subspace of $\mathbb{L}(\mathbb{X},\mathbb{Y}),$ then $T$ is said to be Birkhoff-James orthogonal to $\mathcal{W},$ written as $T\perp_B \mathcal{W},$ if $T\perp_B A$ for all $A\in \mathcal{W}.$

In 1999, Bhatia and \v{S}emrl \cite{BS1} and Paul \cite{Pa} independently characterized $T\perp_B A,$ whenever $T,A\in \mathbb{B}(\mathbb{H}).$  Later on, in \cite{SPM} Sain et. al. generalized the result and characterized $T\perp_B A,$ whenever $T,A\in \mathbb{L}(\mathbb{X},\mathbb{Y}).$ Considering some special subspace $\mathcal{W}$ of $\mathbb{L}(\mathbb{H})$ and some particular operator $T\in \mathbb{L}(\mathbb{H}),$ Andruchow et. al. in \cite{ALRV} obtained characterization for $T\perp_B \mathcal{W}.$ In \cite[Th. 1]{G}, Grover characterized $T\perp_B \mathcal{W},$ whenever $T\in M_n(\mathbb{C})$ and $\mathcal{W}$ is a subspace of $M_n(\mathbb{C}).$ In this paper, our goal is to study  Birkhoff-James orthogonality of an operator $T\in \mathbb{L}(\mathbb{X},\mathbb{Y})$ to any subspace of $\mathbb{L}(\mathbb{X},\mathbb{Y})$ in the setting of arbitrary Banach spaces $\mathbb{X}$ and $\mathbb{Y}.$ We first characterize $T\perp_B \mathcal{W},$ whenever $T\in \mathbb{L}(\mathbb{X},\mathbb{Y})$ and $\mathcal{W}$ is a finite dimensional subspace of $\mathbb{L}(\mathbb{X},\mathbb{Y}),$ where $\mathbb{X}$ is a reflexive Banach space and $\mathbb{Y}$ is a finite dimensional Banach space. If $\mathbb{X}$ and $\mathbb{Y}$ are arbitrary Banach spaces and $\mathcal{W}$ is arbitrary subspace of $\mathbb{L}(\mathbb{X},\mathbb{Y}),$ then we characterize $T\perp_B \mathcal{W}$ under some suitable conditions. We  also characterize Birkhoff-James orthogonality of $T\in \mathbb{L}(\mathbb{H})$ to a subspace of $\mathbb{L}(\mathbb{H})$ in case of infinite dimensional Hilbert space $\mathbb{H}.$ The Theorems \cite[Th. 1]{ALRV}, \cite[Th. 1]{G} and \cite[Th. 2.1(a)]{W2} then follow from our results obtained here.

For an operator $T\in \mathbb{L}(\mathbb{H}),$ the numerical radius of $T,$ denoted by $w(T),$ is defined as $w(T)=\sup\{|\langle Tx,x\rangle|:x\in S_{\mathbb{H}}\}.$ Note that, if $\mathbb{H}$ is a complex Hilbert space, then the numerical radius $w(.)$ of an operator defines a norm on $\mathbb{L}(\mathbb{H}).$ In the space $(\mathbb{L}(\mathbb{H}),w(.)),$ we obtain a sufficient condition for Birkhoff-James orthogonality of an operator to a subspace.

 In due course of time we will see that the norm attainment set of an operator plays an important role in characterizing Birkhoff-James orthogonality of operators. For $T\in \mathbb{L}(\mathbb{X},\mathbb{Y}),$ the set of all unit vectors at which $T$ attains norm is denoted by $M_T,$ i.e., $M_T=\{x\in S_{\mathbb{X}}:\|Tx\|=\|T\|\}.$ We also need the notions of semi-inner-product \cite{Gi,L} and $M-$ideal \cite{HWW}, mentioned below.

\begin{definition}
	Let $ \mathbb{X} $ be a normed space. A function $ [ ~,~ ] : \mathbb{X} \times \mathbb{X} \to K $ is a semi-inner-product (s.i.p.) if for any $ \alpha,~\beta \in K $ and for any $ x,~y,~z \in \mathbb{X}, $ it satisfies the following:\\
	$ (a) $ $ [\alpha x + \beta y, z] = \alpha [x,z] + \beta [y,z], $\\
	$ (b) $ $ [x,x] > 0, $ whenever $ x \neq 0, $\\
	$ (c) $ $ |[x,y]|^{2} \leq [x,x] [y,y], $\\
	$ (d) $ $ [x,\alpha y] = \overline{\alpha} [x,y]. $
\end{definition} 

From \cite{Gi} it follows that in each normed space $\mathbb{X},$ there is a s.i.p. $[,]$ such that for all $x\in \mathbb{X}, ~[x,x]=\|x\|^2.$ In this case, the corresponding s.i.p. is said to be compatible with the norm. In general, there can be many s.i.p. compatible with the norm. In a smooth normed space, there is a unique s.i.p. compatible with the norm. In a Hilbert space, the only s.i.p. is the inner product itself. 

If for a linear projection $P$ on $\mathbb{X},$ $\|x\|=\|Px\|+\|x-Px\|$ holds for all $x\in \mathbb{X},$ then $P$ is called an $L-$projection.  A closed subspace $J$ of $\mathbb{X}$ is called an $L-$summand if $J$ is the range of an $L-$projection and $J$ is called an  $M-$ideal if $J^0$ is an $L-$summand of $\mathbb{X}^*,$ where $J^0=\{f\in \mathbb{X}^*:f(x)=0 ~\forall ~x\in J\}.$ For further information on  $M-$ideals, one can go through \cite{HWW}.

\section{Main results}
We begin this section with the following three known  lemmas. 

\begin{lemma}\cite[Lemma 1.1, pp. 166]{S}\label{lemma-singer}
	Let $\mathbb{X}$ be a Banach space of finite dimension $k$ and $f\in S_{\mathbb{X}^*}.$ Then there exist $f_1,f_2,\ldots,f_h\in Ext(B_{\mathbb{X}^*}),$ where $1\leq h\leq k$ if the scalars are real and $1\leq h\leq 2k-1$ if the scalars are complex and scalars $\lambda_1,\lambda_2,\ldots,\lambda_h>0$ such that $\sum_{i=1}^h\lambda_i=1$ and $f=\sum_{i=1}^h\lambda_if_i.$
\end{lemma} 

\begin{lemma}\cite[Lemma 1.1, pp. 168]{S}\label{lemma-singer2}
	Let $\mathbb{X}$ be a Banach space and $\mathcal{Z}$ is a subspace of $\mathbb{X}.$ Let $f\in Ext(B_{\mathcal{Z}^*}).$ Then there exists $g\in Ext(B_{\mathbb{X}^*})$ such that $g|_{\mathcal{Z}}=f.$
\end{lemma}

\begin{lemma}\cite[Lemma 3.1]{W}\label{lemma-wojcik}
	Suppose that $\mathbb{X}$ is a reflexive Banach space and $\mathbb{Y}$ is a Banach space. Suppose that $\mathbb{K}(\mathbb{X},\mathbb{Y})$ is an $M-$ideal in $\mathbb{L}(\mathbb{X},\mathbb{Y}).$ Let $T\in \mathbb{L}(\mathbb{X},\mathbb{Y}), \|T\|=1$ and  dist$(T,\mathbb{K}(\mathbb{X},\mathbb{Y}))<1.$  Then $M_T\cap Ext(B_\mathbb{X})\neq \emptyset$ and 
	\[Ext ~J(T)=\{y^*\otimes x\in \mathbb{K}(\mathbb{X},\mathbb{Y})^*:x\in M_T\cap Ext(B_{\mathbb{X}}), y^*\in Ext ~J(Tx)\},\]
	where  $y^*\otimes x: \mathbb{K}(\mathbb{X},\mathbb{Y})\to K$ is defined by $y^*\otimes x(S)=y^*(Sx)$ for every $S\in \mathbb{K}(\mathbb{X},\mathbb{Y}).$
\end{lemma}

Now, we obtain a characterization of Birkhoff-James orthogonality of an operator $T\in \mathbb{L}(\mathbb{X},\mathbb{Y})$  to a finite dimensional subspace of $ \mathbb{L}(\mathbb{X},\mathbb{Y}),$ whenever $\mathbb{X}$ is a reflexive Banach space and $\mathbb{Y}$ is a finite dimensional Banach space. Observe that in this case $\mathbb{L}(\mathbb{X},\mathbb{Y})=\mathbb{K}(\mathbb{X},\mathbb{Y}).$
\begin{theorem}\label{th-banachfi}
Suppose $\mathbb{X}$ is a reflexive Banach space and $\mathbb{Y}$ is a finite dimensional Banach space. Suppose $T\in S_{\mathbb{L}(\mathbb{X},\mathbb{Y})}$ and $\mathcal{W}$ is a finite dimensional subspace of $\mathbb{L}(\mathbb{X},\mathbb{Y}).$ Then  $T\perp_B \mathcal{W}$ if and only if there exist $x_i\in M_T\cap Ext(B_{\mathbb{X}}),~y_i^*\in Ext(J(Tx_i))$ and $\lambda_i>0$ for $1\leq i\leq h,(h\in \mathbb{N})$ such that $\sum_{i=1}^h \lambda_i=1$ and $\sum_{i=1}^h \lambda_i y_i^*(Ax_i)=0$ for all $A\in \mathcal{W}.$
\end{theorem}
\begin{proof}
	First we prove the necessary part of the theorem. Suppose $T\perp_B \mathcal{W}.$ Let $\mathcal{Z}=\text{Span}\{T,\mathcal{W}\}.$ Define a linear functional $f:\mathcal{Z}\to K$ by $f(\alpha T+A)=\alpha,$ where $\alpha \in K$ and $A\in \mathcal{W}.$ It is easy to observe that $f\in S_{\mathcal{Z}^*}.$ Now, by Lemma \ref{lemma-singer}, we get $\phi_1,\phi_2,\ldots,\phi_h\in Ext(B_{\mathcal{Z}^*})$ and $\lambda_i>0$ for $1\leq i\leq h$ such that $\sum_{i=1}^h \lambda_i=1$ and $f=\sum_{i=1}^h \lambda_i\phi_i.$ Again using Lemma \ref{lemma-singer2}, we get $\psi_i\in Ext(B_{\mathbb{L}(\mathbb{X},\mathbb{Y})^*})$ for $1\leq i\leq h$ such that $\psi_i|_{\mathcal{Z}}=\phi_i.$ Now $f(T)=1\Rightarrow \sum_{i=1}^h\lambda_i\phi_i(T)=1\Rightarrow \sum_{i=1}^h\lambda_i\psi_i(T)=1\Rightarrow \psi_i(T)=1~\text{for all~} 1\leq i\leq h.$ Therefore, $\psi_i\in J(T)$ and so $\psi_i\in Ext(J(T))$ for all $1\leq i\leq h.$ By Lemma \ref{lemma-wojcik}, there exist $x_i\in M_T\cap Ext(B_{\mathbb{X}})$ and $y_i^*\in Ext(J(Tx_i))$ such that $\psi_i=y_i^*\otimes x_i.$ Now, for each $A\in \mathcal{W}, ~f(A)=0\Rightarrow \sum_{i=1}^h\lambda_i\phi_i(A)=0\Rightarrow \sum_{i=1}^h\lambda_i\psi_i(A)=0\Rightarrow \sum_{i=1}^h\lambda_i y_i^*(Ax_i)=0.$ This completes the proof of the necessary part of the theorem. 

The sufficient part follows easily, since $\sum_{i=1}^h\lambda_i y_i^*(Ax_i)=0\Rightarrow \sum_{i=1}^h\lambda_i y_i^*\otimes x_i(A)=0$ for each $A\in \mathcal{W}$ and $y_i^*\otimes x_i\in Ext(J(T))\Rightarrow \sum_{i=1}^h\lambda_i y_i^*\otimes x_i\in J(T).$ Let $f=\sum_{i=1}^h\lambda_i y_i^*\otimes x_i.$ Then $\|T+A\|\geq |f(T+A)|=|f(T)|=1$ for all $A\in \mathcal{W}.$ Thus, $T\perp_B \mathcal{W}.$ This completes the proof of the theorem. 
\end{proof}
In addition, if we consider that $\mathbb{Y}$ is a smooth Banach space, then we get the following corollary.  
\begin{cor}\label{cor-finite}
Suppose $\mathbb{X}$ is a reflexive Banach space and $\mathbb{Y}$ is a finite dimensional smooth Banach space. Let $T\in S_{\mathbb{L}(\mathbb{X},\mathbb{Y})}$ and  $\mathcal{W}$ be a finite dimensional subspace of $\mathbb{L}(\mathbb{X},\mathbb{Y}).$ Then $T\perp_B \mathcal{W}$ if and only if there exist $x_i\in M_T\cap Ext(B_{\mathbb{X}})$ and $\lambda_i>0$ for $1\leq i\leq h$ such that $\sum_{i=1}^h\lambda_i=1$ and $\sum_{i=1}^h\lambda_i [Ax_i,Tx_i]=0$ for all $A\in \mathcal{W}.$ 	
\end{cor}
\begin{proof}
	By Theorem \ref{th-banachfi}, $T\perp_B \mathcal{W}$ if and only if there exist $x_i\in M_T\cap Ext(B_{\mathbb{X}}),y_i^*\in Ext(J(Tx_i))$ and $\lambda_i>0$ for $1\leq i\leq h$ such that $\sum_{i=1}^h \lambda_i=1$ and $\sum_{i=1}^h \lambda_i y_i^*(Ax_i)=0$ for all $A\in \mathcal{W}.$ Since $\mathbb{Y}$ is smooth, and $y_i^*\in Ext(J(Tx_i)),$ we have $y_i^*(z)=[z,Tx_i]$ for all $z\in \mathbb{Y}.$ Therefore, $\sum_{i=1}^h \lambda_i y_i^*(Ax_i)=0\Rightarrow \sum_{i=1}^h \lambda_i[Ax_i,Tx_i]=0$ for all $A\in \mathcal{W}.$ This proves the result.
\end{proof}
Next, if we consider $\mathbb{X}=\mathbb{Y}=K^n,$ equipped with the usual inner product on $K^n$, then  \cite[Th. 1]{G}  follows from Theorem \ref{th-banachfi}. We prove this in the following corollary.

\begin{cor}
	Let $\mathcal{W}$ be a subspace of $M_n(K)$ and $T\in M_n(K)$ such that $\|T\|=1.$ Then $T\perp_B \mathcal{W}$ if and only if there exists a density matrix $P\in M_n(K)$ such that $T^*TP=P$ and $tr((TP)^*A)=0$ for all $A\in \mathcal{W} $   (note that a positive semidefinite matrix with trace $1$ is called a density matrix.). 
\end{cor}
\begin{proof}
	First we prove the necessary part. Let $T\perp_B \mathcal{W}.$ Then by Corollary \ref{cor-finite}, there exist $x_i\in M_T\cap Ext(B_{\mathbb{X}})$ and $\lambda_i>0$ for $1\leq i\leq h$ such that $\sum_{i=1}^h\lambda_i=1$ and $\sum_{i=1}^h\lambda_i [Ax_i,Tx_i]=0$ for all $A\in \mathcal{W}.$ Since in an inner product space, the inner product is the only s.i.p., we have, $\sum_{i=1}^h\lambda_i \langle Ax_i,Tx_i\rangle=0\Rightarrow \sum_{i=1}^h\lambda_i\langle Tx_ix_i^*,A\rangle=0\Rightarrow \langle TP,A\rangle=0,$ where $P=\sum_{i=1}^h \lambda_i x_ix_i^*.$ Clearly, $P$ is a density matrix. Now, $\langle TP,A\rangle=0\Rightarrow tr((TP)^*A)=0.$ For each $x_i\in M_T, T^*Tx_i=x_i.$ Therefore, $T^*TP=T^*T(\sum_{i=1}^h \lambda_i x_ix_i^*)=\sum_{i=1}^h\lambda_iT^*Tx_ix_i^*=\sum_{i=1}^h\lambda_i x_ix_i^*=P.$
	
	For the sufficient part, suppose there exists a density matrix $P\in M_n(K)$ such that $T^*TP=P$ and $tr((TP)^*A)=0$ for all $A\in \mathcal{W}.$ Since $P$ is a density matrix, there exist $\lambda_i>0$ and $x_i\in K^n,\|x_i\|=1$ for some $1\leq i\leq h$ such that $\sum_{i=1}^h\lambda_i=1$ and $P=\sum_{i=1}^h \lambda_i x_ix_i^*.$  Here $\lambda_i$ is an eigenvalue of $P$ corresponding to the eigenvector $x_i.$ Now, $T^*TP=P\Rightarrow x_i\in M_T$ and $tr((TP)^*A)=0\Rightarrow \sum_{i=1}^h\lambda_i\langle Ax_i,Tx_i\rangle=0.$ Therefore by Corollary \ref{cor-finite}, we have $T\perp_B \mathcal{W}.$ This completes the proof.	
\end{proof}

Suppose $Z$ is a closed subspace of $\mathbb{X}.$ Consider the space $L_Z(\mathbb{X},\mathbb{Y}):=\{T\in \mathbb{L}(\mathbb{X},\mathbb{Y}):T|_Z=0\}.$ In \cite[Th. 2.1(a)]{W2}  W\'{o}jcik obtained a characterization for $T\perp_B L_Z(\mathbb{X},\mathbb{Y})$ for some special operator $T$ assuming that the spaces are real and $Ext(B_{\mathbb{Y}^*})$ is closed. In the following theorem, we show that \cite[Th. 2.1(a)]{W2} also holds for complex Banach spaces and the assumption that $Ext(B_{\mathbb{Y}^*})$ is closed is redundant. 

\begin{theorem}
	Suppose $\mathbb{X}$ is a reflexive Banach space and $\mathbb{Y}$ is an $n-$dimensional Banach space. Let $f\in S_{\mathbb{X}^*}$ and $Z=\ker(f).$ Let $A\in \mathbb{L}(\mathbb{X},\mathbb{Y}),~ T_0\in L_Z(\mathbb{X},\mathbb{Y}).$ Suppose $A_0=A-T_0$ does not attain norm on $Z,$ i.e., $M_{A_0}\cap Z=\emptyset.$ Let 
	\[F_{A_0}=\{y^*\in Ext(B_{\mathbb{Y}^*}):\exists ~x\in M_{A_0}\cap Ext(B_{\mathbb{X}}) \mbox{ with } f(x)>0, y^*(A_0 x)=\|A_0\|\}.\]
	Then $A_0\perp_B L_Z(\mathbb{X},\mathbb{Y})$ if and only if $0\in \mbox{conv}(F_{A_0}).$ 
\end{theorem} 
\begin{proof}
	For convenience, we denote $\mathcal{W}=L_Z(\mathbb{X},\mathbb{Y}).$ Since $\mathbb{Y}$ is finite dimensional, it is easy to observe that $\mathcal{W}$ is finite dimensional. Suppose that $A_0\perp_B \mathcal{W}.$ Without loss of generality we may assume that $\|A_0\|=1.$ Then by Theorem \ref{th-banachfi}, there exist  $x_i\in M_{A_0}\cap Ext(B_{\mathbb{X}}),y_i^*\in Ext(J(A_0x_i))$ and $\lambda_i>0$ for $1\leq i\leq h,(h\in \mathbb{N})$ such that $\sum_{i=1}^h \lambda_i=1$ and $\sum_{i=1}^h \lambda_i y_i^*(Sx_i)=0$ for all $S\in \mathcal{W}.$ Assume that $x\in M_f$ and $f(x)=1.$ Then for each $1\leq i\leq h,~x_i=a_ix+z_i,$ where $a_i\in K$ and $z_i\in Z.$ Without loss of generality, we may assume that $a_i>0$ for all $1\leq i\leq h.$ Then $f(x_i)>0$ and so $y_i^*\in F_{A_0}$ for all $1\leq i \leq h.$ Now, for all $S\in \mathcal{W},~\sum_{i=1}^h \lambda_i y_i^*(Sx_i)=0\Rightarrow \sum_{i=1}^h \lambda_i a_i y_i^*(Sx)=0.$ Let $\{y_1,y_2,\ldots,y_n\}$ be a basis of $\mathbb{Y}.$  For each $1\leq j\leq n,$ define $S_j\in \mathcal{W}$ by $S_j(ax+z)=ay_j,$ where $a\in K,z\in Z.$ Now, $\sum_{i=1}^h \lambda_i a_i y_i^*(S_jx)=0\Rightarrow \sum_{i=1}^h \lambda_i a_i y_i^*(y_j)=0$ for all $1\leq j\leq n.$ Therefore, $\sum_{i=1}^h \lambda_i a_i y_i^*=0.$ Let $\alpha=\sum_{i=1}^{h}\lambda_ia_i>0.$ Then $\sum_{i=1}^{h}\frac{\lambda_ia_i}{\alpha}y_i^*=0\Rightarrow 0\in \mbox{conv}(F_{A_0}).$\\
	 
	Conversely, suppose that $0\in \mbox{conv}(F_{A_0}).$ Then $\sum_{i=1}^h\mu_iy_i^*=0,$ where $\mu_i>0,~\sum_{i=1}^h\mu_i=1$ and $y_i^*\in F_{A_0}.$ Clearly, for each $1\leq i\leq h,$ there exists $x_i\in  M_{A_0}\cap Ext(B_{\mathbb{X}})$  such that  $f(x_i)>0$ and $ y_i^*(A_0 x_i)=\|A_0\|=\|A_0x_i\|.$ Therefore, $y_i^*\in Ext(J(A_0x_i)).$ Now, let $x\in M_f$ and $f(x)=1.$ Let $x_i=a_ix+z_i,$ where $a_i\in K,z_i \in Z.$ Then $f(x_i)>0\Rightarrow a_i>0.$ Let $\alpha=\sum_{i=1}^h\frac{\mu_i}{a_i}>0$ and $\lambda_i=\frac{\mu_i}{a_i\alpha}>0.$ Then for all $S\in \mathcal{W},~\sum_{i=1}^h\lambda_iy_i^*(Sx_i)= \sum_{i=1}^h\frac{\mu_i}{a_i\alpha}y_i^*(Sx_i)=\sum_{i=1}^h\frac{\mu_i}{\alpha}y_i^*(Sx)=\frac{1}{\alpha}(\sum_{i=1}^h\mu_iy_i^*)(Sx)=0.$ Therefore, by Theorem \ref{th-banachfi}, $T\perp_B \mathcal{W}.$  This completes the proof of the theorem. 
\end{proof}

In the following theorem, we generalize \cite[Th. 1]{G}. In fact, in place of $M_n(K),$ we consider $\mathbb{L}(\mathbb{H}),$ where $\mathbb{H}$ is a Hilbert space, not necessary finite dimensional. Here we recall that for every Hilbert space $\mathbb{H},$ $\mathbb{K}(\mathbb{H})$ is always an $M-$ideal in $\mathbb{L}(\mathbb{H}).$

\begin{theorem}\label{th-hil}
	Let $\mathbb{H}$ be a Hilbert space. Suppose $T\in \mathbb{L}(\mathbb{H})$ is such that $\|T\|=1,$ $M_T=S_{H_0},$ where $H_0$ is a finite dimensional subspace of $\mathbb{H}$ and $\|T\|_{H_0^\perp}<\|T\|.$ Then for any subspace $\mathcal{W}$ of $\mathbb{L}(\mathbb{H}),~T\perp_B \mathcal{W}$ if and only if there exist $x_1,x_2,\ldots,x_n\in M_T$ and $\lambda_1,\lambda_2,\ldots,\lambda_n>0$ such that $\sum_{i=1}^n\lambda_i=1$ and  $\sum_{i=1}^n\lambda_i\langle Ax_i,Tx_i\rangle=0$ for all $A\in \mathcal{W}.$ 
\end{theorem}
\begin{proof}
	We prove the theorem for complex Hilbert space. For real Hilbert space, the proof follows similarly. We first prove the necessary part. We claim that $dist(T,\mathbb{K}(\mathbb{H}))<1.$ If possible, suppose that this is not true. Then for every $S\in \mathbb{K}(\mathbb{H}),~dist(T,\text{Span}\{S\})\geq dist (T,\mathbb{K}(\mathbb{H}))\geq 1,$ i.e., for every $\lambda\in K,\|T-\lambda S\|\geq dist(T,\text{Span}\{S\})\geq 1\Rightarrow T\perp_B S.$ Define $S:\mathbb{H}\to\mathbb{H}$ by $Sx=Tx$ whenever $x\in H_0$ and $Sx=0,$ whenever $x\in H_0^\perp.$ Then clearly, $S\in \mathbb{K}(\mathbb{H}),$ since $H_0$ is finite dimensional. Hence, $T\perp_B S.$ By \cite[Th. 3.1]{PSG}, there exists $x\in M_T=S_{H_0}$ such that $Tx\perp Sx\Rightarrow Tx\perp Tx,$ a contradiction. This proves the claim. 
	
	Now, we show that Span $Ext(J(T))$ is finite dimensional. Since $\mathbb{K}(\mathbb{H})$ is an $M-$ideal in $\mathbb{L}(\mathbb{H})$ and $dist(T,\mathbb{K}(\mathbb{H}))<1,$ by Lemma \ref{lemma-wojcik}, $Ext(J(T))=\{y^*\otimes x:x\in M_T,y^*\in Ext (J(Tx))\},$ where $y^*\otimes x (S)=y^*(Sx)=\langle Sx,Tx\rangle$ for every $S\in \mathbb{L}(\mathbb{H}).$ So we can write $Ext(J(T))=\{x\otimes Tx:x\in M_T\},$ where $x\otimes Tx(S)=\langle Sx,Tx\rangle$ for every $S\in \mathbb{L}(\mathbb{H}).$  Suppose that $\{e_1,e_2,\ldots,e_k\}$ is an orthonormal basis of $H_0.$ Then 
	\begin{eqnarray*}
	&&\text{Span}~ Ext(J(T))\\
	&=&\text{Span}~  \{x\otimes Tx:x\in M_T\}\\
	&=&\text{Span}~ \{\sum_{i,j=1}^ka_i\overline{a_j}e_i\otimes Te_j:\sum_{i=1}^k|a_i|^2=1\}\\
	&=& \text{Span}~  \{e_i\otimes Te_j:1\leq i,j\leq k\}.
	\end{eqnarray*}
	Therefore, Span $Ext(J(T))$ is finite dimensional.
	
	Next, we construct $f\in J(T)$ such that $f(A)=0$ for all $A\in \mathcal{W}.$ Let $\mathcal{Z}=\text{Span} \{T,\mathcal{W}\}.$ Define $g:\mathcal{Z}\to \mathbb{C}$ by $g(\alpha T+A)=\alpha,$ where $\alpha \in \mathbb{C}$ and $A\in \mathcal{W}.$ Since $T\perp_B\mathcal{W},~ g\in S_{\mathcal{Z}^*}.$ By Hahn-Banach theorem, there exists $f\in \mathbb{L}(\mathbb{H})^*$ such that $\|f\|=\|g\|=1$ and $f|_{\mathcal{Z}}=g.$ Therefore, $f(T)=1\Rightarrow f\in J(T)$ and $f(A)=0$ for all $A\in \mathcal{W}.$ Since $J(T)$ is weak*compact, convex subset of $\mathbb{L}(\mathbb{H})^*,$ using Krein-Milman Theorem, we get, $J(T)=\overline{\mbox{conv}}^{w*} (Ext( J(T)))\subseteq \overline{\mbox{Span} (Ext(J(T)))}^{w*}= \overline{\mbox{Span} (Ext(J(T)))}=\mbox{Span} (Ext(J(T))),$ since $\mbox{Span} (Ext(J(T)))$ is finite dimensional. Now, using Lemma \ref{lemma-singer}, we get extreme points $f_1,f_2,\ldots,f_n$ of the unit ball of $ \mbox{Span} (Ext(J(T)))$ and $\lambda_1,\lambda_2,$ $\ldots,\lambda_n >0$ such that $\sum_{i=1}^n \lambda_i=1$ and $f=\sum_{i=1}^n\lambda_i f_i.$ It can be easily shown that $f_i\in Ext(J(T))$ for all $1\leq i\leq n.$ So for $1\leq i\leq n,$ there exist $x_i\in M_T$ such that $f_i=x_i\otimes Tx_i.$ Now, for all $A\in \mathcal{W},$ $f(A)=0\Rightarrow \sum_{i=1}^n\lambda_i f_i(A)=0\Rightarrow \sum_{i=1}^n\lambda_i x_i\otimes Tx_i(A)=0\Rightarrow \sum_{i=1}^n\lambda_i \langle Ax_i,Tx_i\rangle=0.$ This completes the proof of the necessary part of the theorem.
	
	For the sufficient part, suppose that $\sum_{i=1}^n\lambda_i\langle Ax_i,Tx_i\rangle=0$ for all $A\in \mathcal{W},$ where $\sum_{i=1}^n\lambda_i=1$ and $\lambda_i>0$ for $1\leq i\leq n.$ Then $\sum_{i=1}^n\lambda_i x_i\otimes Tx_i (A)=0.$ Now, $x_i\otimes Tx_i\in J(T)$ and $J(T)$ is convex, so $\sum_{i=1}^n\lambda_i x_i\otimes Tx_i\in J(T).$ Let $f=\sum_{i=1}^n\lambda_i x_i\otimes Tx_i.$ Then for every $A\in \mathcal{W},\|T+A\|\geq |f(T+A)|=|f(T)|=1\Rightarrow T\perp_B\mathcal{W}.$ This completes the proof of the theorem.  
\end{proof}

We now generalize Theorem \ref{th-banachfi} by considering the range space as an arbitrary Banach spaces instead of taking finite dimensional Banach space. For this purpose, we need to assume some restrictions on the norm attainment set of the operator. Let us recall that  for  $x\in S_{\mathbb{X}},$ if $\mbox{Span}(J(x))$ is finite dimensional and  dimension of $\mbox{Span}(J(x))$ is $m,$ then $x$ is said to be $m-$smooth \cite{KS}.  
\begin{theorem}\label{th-banachinf}
	Suppose $\mathbb{X}$ is a reflexive Banach space and $\mathbb{Y}$ is arbitrary Banach space. Let $\mathbb{K}(\mathbb{X},\mathbb{Y})$ be an $M-$ideal in $\mathbb{L}(\mathbb{X},\mathbb{Y}).$ Let $T\in \mathbb{L}(\mathbb{X},\mathbb{Y})$ be such that the following conditions hold:\\
	(i) $\|T\|=1$ and dist$(T,\mathbb{K}(\mathbb{X},\mathbb{Y}))<1,$ \\
	(ii)  $|M_T\cap Ext(B_{\mathbb{X}})|<\infty,$ \\
	(iii) $Tx$ is $m_x$-smooth for some $m_x<\infty,$ for each $x\in M_T\cap Ext(B_{\mathbb{X}}).$\\
	Then for a subspace $\mathcal{W}$ of $\mathbb{L}(\mathbb{X},\mathbb{Y}),$ $T\perp_B \mathcal{W}$ if and only if  there exist $x_i\in M_T\cap Ext(B_{\mathbb{X}}),y_i^*\in Ext(J(Tx_i))$ and $\lambda_i>0$ for $1\leq i\leq n,(n\in \mathbb{N})$ such that $\sum_{i=1}^n \lambda_i=1$ and $\sum_{i=1}^n \lambda_i y_i^*(Ax_i)=0$ for all $A\in \mathcal{W}.$
\end{theorem}
\begin{proof}
	The sufficient part follows similarly as in Theorem \ref{th-banachfi}. We only prove the necessary part of the theorem. First observe that Span $Ext (J(T))$ is finite dimensional. In fact, Span $Ext(J(T))=$ Span $\{y^*\otimes x:x\in M_T\cap Ext(B_{\mathbb{X}}), y^*\in Ext(J(Tx))\}.$ Now, $(ii)$ and $(iii)$ imply that Span $Ext (J(T))$ is finite dimensional. Let $\mathcal{Z}=Span \{T,\mathcal{W}.\}$ Now, as in Theorem \ref{th-banachfi}, we define a linear functional $g$ on $\mathcal{Z}$ by $g(\alpha T+A)=1,$ where $\alpha \in K$ and $A\in \mathcal{W}.$ Using $T\perp_B \mathcal{W},$ it can be shown that $\|g\|=1.$ Therefore, by Hahn-Banach theorem, there exists $f\in \mathbb{L}(\mathbb{X},\mathbb{Y})^*$ such that $\|f\|=1$ and $f|_{\mathcal{Z}}=g.$ Clearly, $f\in J(T).$ Since $J(T)$ is weak*compact, convex subset of $\mathbb{L}(\mathbb{X},\mathbb{Y})^*,$ using Krein-Milman Theorem, we get, $J(T)=\overline{\mbox{conv}}^{w*} (Ext( J(T)))\subseteq \overline{\mbox{Span} (Ext(J(T)))}^{w*}= \overline{\mbox{Span} (Ext(J(T)))}=\mbox{Span} (Ext(J(T))),$ since $\mbox{Span} (Ext(J(T)))$ is finite dimensional. Now, proceeding similarly as in Theorem \ref{th-hil}, we get $f_1,f_2,\ldots,f_n\in Ext(J(T))$ and $\lambda_1,\lambda_2,\ldots,\lambda_n>0$ such that $\sum_{i=1}^n \lambda_i=1$ and $f=\sum_{i=1}^n\lambda_i f_i.$ By Lemma \ref{lemma-wojcik}, $f_i=y_i^*\otimes x_i,$ where $x_i\in M_T\cap Ext(B_{\mathbb{X}})$ and $y_i^*\in Ext(J(Tx_i))$ for each $1\leq i\leq n.$ Now, for each $A\in \mathcal{W}, f(A)=g(A)=0.$ Thus, $\sum_{i=1}^n\lambda_i f_i(A)=0\Rightarrow \sum_{i=1}^n\lambda_i y_i^*\otimes x_i (A)=0\Rightarrow \sum_{i=1}^n \lambda_i y_i^* (Ax_i)=0.$ This completes the proof of the theorem.
\end{proof}

\begin{remark}
	Let $\mathbb{H}$ be a Hilbert space and $T,A\in \mathbb{L}(\mathbb{H})$ such that $0\notin \sigma_{\text{app}}A.$ Following the notation of \cite{P}, we say that 
	\[M_T(A)=\sup_{\|x\|=1}\{\|Tx-\frac{\langle Tx,Ax\rangle}{\|Ax\|^2}Ax\|\}.\]
	Then by \cite[Th. 4]{P}, $M_T(A)^2=\sup\{g(T^*T)-\frac{|g(A^*T)|^2}{g(A^*A)}:g \text{ is a state and } g(A^*A)\neq 0\}.$ From \cite[Th. 4]{PHD}, $M_T(A)=dist(T,\text{Span}\{A\}).$ Therefore, if $g$ is a state such that $g(A^*A)\neq 0,$ then 
	\begin{equation}\label{eq-p}
	g(T^*T)-\frac{|g(A^*T)|^2}{g(A^*A)}\leq dist(T,\text{Span}\{A\})^2.
	\end{equation}
	In \cite[Th. 3.3]{BS}, Bhatia and Sharma proved that for any positive unital linear map $\phi:M_n(K)\to M_k(K)$ and $A\in M_n(K),$
	\begin{equation}\label{eq-bs}
	\phi(A^*A)-\phi(A)^*\phi(A)\leq dist(T,\text{Span}\{I\})^2.
	\end{equation}
	In the same paper, considering different positive unital linear maps the authors obtained many lower bounds for the distance of a matrix $T$ form Span$\{I\},$ using Equation \ref{eq-bs}.
	
	Now, considering different states on $\mathbb{L}(\mathbb{H}),$ we obtain lower bounds for distance of an operator $T$ from Span$\{A\},$ using Equation \ref{eq-p}. For example, suppose that $\mathbb{H}=\ell_2^n,$ i.e., $\mathbb{L}(\mathbb{H})=M_n(K).$
	
	$(i)$ Let each row of $A$ is non-zero. Then for each $1\leq j\leq n,$ consider $g_j:M_n(K)\to K$ defined by $g_j(S)=s_{jj},$ where $S=(s_{ij})\in M_n(K).$ Then 
	\[dist(T,\text{Span}\{A\})^2\geq \sum_{i=1}^n|t_{ij}|^2-\frac{\sum_{i=1}^n|\overline{t_{ij}}a_{ij}|^2}{\sum_{i=1}^n|a_{ij}|^2}.\] 
	Therefore,
		\[dist(T,\text{Span}\{A\})^2\geq \max_{1\leq j\leq n}\Big\{ \sum_{i=1}^n|t_{ij}|^2-\frac{\sum_{i=1}^n|\overline{t_{ij}}a_{ij}|^2}{\sum_{i=1}^n|a_{ij}|^2}\Big\}.\]  
	$(ii)$ Similarly, if $\|A\|_F\neq 0,$ where $\|A\|_F^2=tr(A^*A),$ and we consider  $g:\mathbb{L}(\mathbb{H})\to K$ defined by $g(S)=\frac{1}{n}tr(S),$ then we get 
	\[dist(T,\text{Span}\{A\})^2\geq \frac{1}{n}\|T\|_F^2-\frac{1}{n} \frac{|tr(T^*A)|^2}{\|A\|^2_F}.\]
\end{remark}

In \cite{MPS}, the authors obtained characterization of ``numerical radius orthogonality ($\perp_w$)", i.e., Birkhoff-James orthogonality in $\mathbb{L}(\mathbb{H})$ with respect to numerical radius norm. Here we obtain a sufficient condition for numerical radius orthogonality of an operator in $\mathbb{L}(\mathbb{H})$ to a subspace of $\mathbb{L}(\mathbb{H}).$ If $f\in (\mathbb{L}(\mathbb{H}),w(.))^*,$ then we denote the norm of $f$ by $\|f\|_w.$

\begin{theorem}\label{th-suffnu}
	Suppose $\mathbb{H}$ is a complex Hilbert space and $T(\neq 0)\in \mathbb{L}(\mathbb{H}).$ Let $\mathcal{W}$ be a subspace of $\mathbb{L}(\mathbb{H}).$ Let there exist $x_1,x_2,\ldots,x_n\in S_{\mathbb{H}}$ and $\lambda_1,\lambda_2,\ldots,\lambda_n>0$ such that $|\langle Tx_i,x_i\rangle|=w(T)$ for all $1\leq i\leq n,$ $\sum_{i=1}^n\lambda_i=1$ and 
	\[\sum_{i=1}^n\lambda_i\overline{\langle Tx_i,x_i\rangle}\langle Ax_i,x_i\rangle=0 ~\text{for all }A\in \mathcal{W}.\]
	Then $T\perp_w \mathcal{W}.$
\end{theorem}
\begin{proof}
	For each $1\leq i\leq n,$ define linear functional $f_i:\mathbb{L}(\mathbb{H})\to \mathbb{C}$ as follows:
	\[f_i(S)=\frac{1}{w(T)}\overline{\langle Tx_i,x_i\rangle}\langle Sx_i,x_i\rangle, \text{ for all }S\in \mathbb{L}(\mathbb{H}).\]
	Clearly, $|f_i(S)|\leq w(S)$ for all $S\in \mathbb{L}(\mathbb{H}),$ i.e., $\|f_i\|_w\leq 1.$ Now, $|f_i(T)|=w(T)\Rightarrow \|f_i\|_w=1.$ Consider the set 
	\[J_w(T)=\{f\in (\mathbb{L}(\mathbb{H}),w(.))^*:\|f\|_w=1 ~ \&~ f(T)=w(T) \}.\]
	$J_w(T)$ is a convex set and $f_i\in J_w(T)$ for all $1\leq i\leq n.$ Therefore, $\sum_{i=1}^{n}\lambda_i f_i\in J_w(T).$ Let $f=\sum_{i=1}^{n}\lambda_i f_i.$ Then for $A\in \mathcal{W},$ 
	$$f(A)=\sum_{i=1}^{n}\lambda_i f_i(A)=\frac{1}{w(T)}\sum_{i=1}^{n}\lambda_i \overline{\langle Tx_i,x_i\rangle}\langle Ax_i,x_i\rangle=0.$$
	Now, for each $\lambda\in \mathbb{C},~w(T+\lambda A)\geq |f(T+\lambda A)|=|f(T)|=w(T).$ Hence $T\perp_wA\Rightarrow T\perp_w\mathcal{W}.$ This completes the proof of the theorem.
\end{proof}

\bibliographystyle{amsplain}

\end{document}